\documentclass[12pt,reqno]{amsart}
 
\usepackage{fullpage} 
\usepackage{amssymb}
\usepackage{amsmath} 
\usepackage{amsthm}
\usepackage{color}
\usepackage{graphicx}

\newtheorem{theorem}{Theorem}[section]
\newtheorem{lemma}[theorem]{Lemma}
\newtheorem{proposition}[theorem]{Proposition}
\newtheorem{corollary}[theorem]{Corollary}
\newtheorem{conjecture}[theorem]{Conjecture}

\theoremstyle{definition}
\newtheorem{question}[theorem]{Question}
\newtheorem{definition}[theorem]{Definition}
     
\newtheorem{remark}[theorem]{Remark}

\newcommand{\acts}[1]{\stackrel{#1}{\curvearrowright}}

\newcommand{\arrow}{\rightarrow}
\newcommand{\Hom}{\operatorname{Hom}}
\newcommand{\N}{\mathbb{N}}
\newcommand{\Z}{\mathbb{Z}}
\newcommand{\U}{\mathcal{U}}
\newcommand{\W}{\mathcal{W}}
\renewcommand{\O}{\mathcal{O}}
\newcommand{\Sub}{\operatorname{Sub}}
\newcommand{\Env}{\operatorname{Env}}
\newcommand{\Is}{\operatorname{Is}}
\newcommand{\fg}{\operatorname{fg}}
\newcommand{\fii}{\operatorname{fi}}
\newcommand{\Occ}{\operatorname{Occ}}
\newcommand{\Sch}{\operatorname{Sch}}
\newcommand{\Cam}{\operatorname{CoAm}}
\newcommand{\G}{\mathcal{G}}

\renewcommand{\subset}{\subseteq}
\newcommand{\la}{\langle}
\newcommand{\ra}{\rangle}

\newcommand{\minus}{\smallsetminus}
\renewcommand{\epsilon}{\varepsilon}
\renewcommand{\phi}{\varphi}
\newcommand{\trace}[2]{\text{trace}_{#2}\text{($#1$)}}
\newcommand{\rest}[1]{\big|_{#1}}

\author{Yair Glasner \and Daniel Kitroser \and Julien Melleray}
\title{From isolated subgroups to generic permutation representations} 

\date{}

\begin{document}
\begin{abstract}
Let $G$ be a countable group, $\Sub(G)$ the (compact, metric) space of all subgroups of $G$ with the Chabauty topology and $\Is(G) \subset \Sub(G)$ the collection of isolated points. We denote by $X!$ the (Polish) group of all permutations of a countable set $X$. Then the following properties are equivalent: (i) $\Is(G)$ is dense in $\Sub(G)$, (ii) $G$ admits a ``generic permutation representation''. Namely there exists some $\tau^* \in \Hom(G,X!)$ such that the collection of permutation representations $\{\phi \in \Hom(G,X!) \ | \ \phi {\text{ is permutation isomorphic to }} \tau^*\}$ is co-meager in $\Hom(G,X!)$. We call groups satisfying these properties {\it{solitary}}. Examples of solitary groups include finitely generated LERF groups and groups with countably many subgroups. 
\end{abstract}
\maketitle

\section{Introduction}

Let $G$ be a countable group and $\Sub(G)$, the space of all subgroups of $G$ endowed with the Chabauty topology, which makes it into a compact metrizable totally disconnected space. The easiest way to define this topology is to embed $\Sub(G) \subset \{0,1\}^G$ as a closed subset and induce the Tychonoff topology on $\{0,1\}^G$. The group $G$ acts on $\Sub(G)$ continuously by conjugation $g \cdot \Delta = g \Delta g^{-1}$. One is naturally led to the question of how the structure of the topological space $\Sub(G)$, or more generally the topological dynamical system $(G,\Sub(G))$, is reflected in the algebraic structure of $G$. 

The Cantor-Bendixon structure theory of compact spaces leads us to consider the decomposition
$\Sub(G) = \Is(G) \sqcup \Sub(G)'.$
Here $\Is(G)$ is the collection of isolated points, or {\it{isolated subgroups}} as we shall refer to them, and $\Sub(G)'$ is its  complement. Isolated subgroups are special from the algebraic point of view! Clearly there are only countably many of them. One can think of isolated subgroups in algorithmic terms as subgroups that are {\it{detectable}} or {\it{recognizable}} via a finite algorithmic procedure. A subgroup $\Delta \in \Sub(G)$ is isolated if and only if it can be identified by making a finite number of membership and non-membership tests for specific elements. It is sometimes convenient to think of this in terms of Schreier graphs. Let $S$ be a symmetric generating set for $G$. A subgroup $\Delta \in \Sub(G)$ is isolated if one can find a finite algorithm that would recognize the Schreier graph $\Sch(G,\Delta,S)$ out of all Schreier graphs of the group $G$. Note that $S$ might very well be infinite, and consequently the Schreier graphs in question may fail to be locally finite. Still the algorithm is allowed to look only at finitely many edges. From these characterizations it is easy to see that isolated groups are always finitely generated. In the special case where $G$ itself is finitely generated every finite index subgroup is isolated and we obtain inclusions $\Sub^{\fii}(G) \subset \Is(G) \subset \Sub^{\fg}(G)$ where $\Sub^{\fii}(G), \Sub^{\fg}(G)$ stand for finite index and finitely generated subgroups respectively. It is clear that $\Is(G)$ is always a discrete countable open subset of $\Sub(G)$. Our main new definition is the following: 
\begin{definition}
A group $G$ is called {\it{solitary}} if the isolated points $\Is(G)$ are dense in $\Sub(G)$.
\end{definition}

Let $X$ be a countable set and $X!$ the group of all permutations of $X$. The topology of pointwise convergence makes $X!$ into a Polish group: separable, metrizable and complete.  The space $\Hom(G,X!)$ is the space of all permutation representations of $G$ and is clearly also a Polish space. There is a natural action 
\begin{align} \label{eqn:action}
&X! \times \Hom(G,X!)  \rightarrow \Hom(G,X!) \\
\nonumber &(\alpha, \phi) \mapsto \alpha\cdot\phi: g \mapsto \alpha \phi(g) \alpha^{-1}
\end{align} 
Two permutation representations that are in the same orbit are said to be \emph{isomorphic as permutation representations}. We will be interested in Baire generic properties of permutation representations and in particular in the existence of a generic permutation representation in the sense of the following: \begin{definition}
The group $G$ is said to have a {\it{generic permutation representation}} if there is a permutation representation $\tau^* \in \Hom(G,X!)$ whose orbit 
\begin{eqnarray*}
X!(\tau^*) & = & \{\alpha \cdot \tau^* \ | \ \alpha \in X!\} \\
& = & \{\phi \in \Hom(G,X!) \ | \ \phi {\text{ is a permutation isomorphic to }} \tau^{*}\}
\end{eqnarray*}
is co-meager in $\Hom(G,X!)$. 
\end{definition}

It turns out that the existence of a generic permutation representation is captured by the structure of the topological space $\Sub(G)$.
\begin{theorem} \label{thm:main} (Main theorem)
A countable group $G$ admits a generic permutation representation if and only if it is solitary. 
\end{theorem}
\begin{definition}
A group $G$ is called \emph{subgroup separable} or \emph{locally extended residually finite} (\emph{LERF} for short), if every finitely generated subgroup of $G$ is the intersection of finite index subgroups. Or equivalently if any finitely generated subgroup is closed in the profinite topology on $G$. 
\end{definition}

Examples of LERF groups include finitely generated abelian groups, free  groups \cite{Hall:LERF}, surface  groups  \cite{Scott:LERF_corr} and more generally, limit groups \cite{Wilton:LERF}, the Grigorchuk group \cite{Grigorchuk_Wilson}, many lamplighter groups \cite{GK:lamplighter_subgps}. Recently the LERF property attracted a lot of attention as Agol's proof of the LERF property for the fundamental group of a closed hyperbolic 3-manifold \cite{Agol:virtual_haken} was a central ingredient in his solution to Thruston's virtual Haken conjecture. 

The following theorem is analogous to our main theorem above. It shows in particular that finitely generated LERF groups are solitary. 

\begin{theorem}\label{thm:LERF}
Let $G$ be a finitely generated group, then the following conditions are equivalent:
\begin{enumerate}
\item $G$ is LERF.
\item The collection $\Sub^{\fii}(G)$ of finite index subgroups is dense in $\Sub(G)$.
\item \label{itm: typical_action} $G$ has a generic permutation representation, all of whose orbits are finite. 
\end{enumerate}
\end{theorem}
\begin{remark}
The fact that the first and third condition above are equivalent also follows from earlier work of Rosendal: Proposition 8(B) in \cite{Rosendal1} proves that $G$ is LERF if and only if a generic permutation representation 
has only finite orbits, while Theorem 11 of \cite{Rosendal2} shows that a finitely generated LERF group admits a generic permutation representation. Note that Theorem 11 of \cite{Rosendal2} is only formulated for groups 
acting by isometries on the rational Urysohn space, but see the remark in the last paragraph of \cite{Rosendal2}, where it is pointed out that the proofs adapt to other metric spaces, notably the Urysohn space with distances 
$\{0,1\}$, that is, an infinite countable set.
\end{remark}

In the paper we state and prove a more detailed version of this theorem that holds also for countable groups. Once we leave the realm of finitely generated groups isolated subgroups and solitary groups no longer generalize finite index subgroups and LERF groups respectively. Each of these theories goes in its own way. Our impression is that in some settings the choice of isolated subgroups and solitary groups is the more natural one. The following theorem summarizes some examples and structural results we have about solitary groups.
\begin{theorem} \label{thm:solitary_prop} Some properties of solitary groups.
\begin{enumerate}
\item \label{itm:l->s} Finitely generated LERF groups are solitary.
\item \label{itm:countable} If $\Sub(G)$ is countable then $G$ is solitary.
\item \label{itm:fg_kernel} Let $1 \arrow N \arrow H \arrow G \arrow 1$ be a short exact sequence of countable groups such that $N$ is finitely generated, as an abstract group. If $H$ is solitary then so is $G$.  
\item \label{itm:free_prod} The free product of two countable groups $G*H$ is solitary if and only if one of the following two options hold:
\begin{itemize}
\item both $G$ and $H$ are LERF and finitely generated. 
\item $G$ is solitary and $H$ is trivial, or vice-versa.
\end{itemize} 
\end{enumerate}
\end{theorem}
The situation considered in condition (\ref{itm:fg_kernel}) above is identical to the one appearing in the famous Rips construction. This naturally leads to the following:
\begin{question}
Is it true that every finitely generated solitary group $G$ can be placed in a short exact sequence $1 \arrow N \arrow H \arrow G \arrow 1$ where $H$ is solitary and hyperbolic and $N$ is finitely generated as an abstract group.
\end{question}

Theorem \ref{thm:main} is tightly connected to the notion of ample generics in Polish groups. We adopt the notation of \cite{MT:ample_generics} ; \cite{Rosendal2} was the first paper to express this property in term of generic orbits on ``presentation varieties''.
\begin{definition}
We say that a Polish group $P$ has \emph{ample generics}, if $\Hom(F_n,P)$ admits a generic $P$ orbit for every $n \in \N$. 
\end{definition} 
The notion of ample generics was first introduced in \cite{HHLS:small_index} in order to study the small index property in Polish groups. Namely that every subgroup of index $< 2^{\aleph_0}$ is open. Indeed this and additional consequences such as automatic continuity of abstract homomorphisms into any separable group were subsequently established for all groups with ample generics in \cite{Kechris_Rosendal}. We refer the readers to all of the above mentioned papers (see also the survey \cite{Glasner_Weiss:Rokhlin}) and the references therein. With this terminology in place Theorem \ref{thm:LERF} shows that the following two well known facts:
\begin{itemize}
\item $X!$ has ample generics,
\item finitely generated free groups $F_n$ are LERF,
\end{itemize}
are in fact two different realizations of the same phenomenon. In both cases one seeks a generic $P$-orbit in $\Hom(G,P)$. But in the study of groups with ample generics, one fixes $G$ (or more precisely lets $G$ range over all finitely generated free groups) and lets the Polish group $P$ vary; whereas in the study of solitary groups we fix the Polish group $P = X!$ and consider the class of all the countable groups $G$ that give rise to a generic orbit. In view of the very natural characterization that arises from Theorem \ref{thm:LERF} the following question seems natural:
\begin{question}
Given a Polish group $P$, describe the class of all finitely generated groups $G$ for which $\Hom(G,P)$ has a generic $P$-orbit. 
\end{question} 

In particular, the answer should contain all finitely generated free groups whenever $P$ has ample generics. The group $X!$ is probably the simplest example of a Polish group with ample generics, but there are many others. 

We now turn to generalizing Theorem \ref{thm:LERF}, and the notion of LERF groups, in a different direction. A subgroup $H \leq G$ is called \emph{co-amenable} if there is a $G$-invariant mean on $G/H$ (see also Definitions \ref{def:amenable_action}, \ref{def:coam}). Co-amenable subgroups generalize finite index subgroups in much the same way that amenable groups generalize finite groups. In view of that and of Theorems \ref{thm:LERF} and \ref{thm:main} we can generalize the notion of LERF groups as follows:

\begin{definition}
A group $G$ is \emph{amenably separable}, or \emph{A-separable} for short, if the set of co-amenable subgroups of $G$ is dense in $\Sub(G)$.
\end{definition} 

In view of the Theorem \ref{thm:LERF}, every LERF group is $A$-separable. Another obvious source for examples is the class of all amenable groups. These are $A$-separable since all of their subgroups are co-amenable. In chapter \ref{sec:Asep} of this work, we initiate the study of A-separable groups. Our hope is that the notion of A-separability will prove to be a useful generalization of the, a-priori very different, properties of LERF and amenability. This situation is perhaps reminiscent of the way in which sofic groups simultaneously generalize the notions of residual finiteness and amenability. In these terms the analogue of Theorem \ref{thm:main} is the following

\begin{theorem}\label{Asep thm}
A countable group $G$ is A-separable if and only if for a generic action of $G$ on a countable set, the action on every orbit is amenable.
\end{theorem}
\noindent Here are some properties of A-separable groups.
\begin{theorem} \label{basic_prop}
The following properties hold for the class of A-separable groups: 
\begin{itemize}
\item LERF groups and amenable groups are A-separable.
\item The class of A-separable groups is closed under free products.
\item There exist A-separable groups which are neither LERF nor amenable.
\item A group with property (T) is A-separable if and only if it is LERF.
\item Higher rank lattices in non-compact simple Lie groups which satisfy the congruence subgroup property are never A-separable. 
\end{itemize}
\end{theorem}

The paper is arranged as follows. Section \ref{sec:genprop} is dedicated to a systematic investigation of the topological spaces $\Sub(G)$, $\Hom(G,X!)$ and the standard stabilizer map $\Hom(G,X!) \arrow \Sub(G)$ between them. In Section \ref{sec:LERF} we prove Theorem \ref{thm:LERF}. Chapter \ref{sec:genact} is dedicated to solitary groups and there we prove Theorems \ref{thm:main} and \ref{thm:solitary_prop}. Finally, in chapter \ref{sec:Asep} we prove Theorems \ref{Asep thm} and \ref{basic_prop} and give examples of non amenable, non LERF groups that are A-separable. The results in this work also to appear as part of the Ph.D. dissertation of the second author \cite{Kitroser:thesis}. 

\section{Dense and generic properties of actions and subgroups}\label{sec:genprop}

\subsection{The space of permutation representations $\mathbf{\Hom(G,X!)}$}
Let $X$ be a countable set and $X!$ the full symmetric group of all bijections of $X$ onto itself. We endow $X!$ with the topology of pointwise convergence which makes it into a Polish topological group. In other words a topological group that is separable, and admits a complete metric. The latter fact is important for us because it shows that $X!$ is a Baire space; though we will never consider any specific metric. An explicit basis for the topology can be given by the sets
\[ \U(\alpha,A) := \left\{\beta\in X!\mid \beta\rest{A} = \alpha\rest{A}\right\}\quad (\alpha\in X!, A\subset X\text{ finite}).\]
On $X!^{n}$ we will always put the product topology, which is still Polish for every $n \in \N \cup \{\infty\}$. 

Let $G$ be a countable group with a given presentation $G = \la S\mid R\ra$ where $S = \{s_1,s_2,\dots\}$. Then, we can identify $\Hom(G,X!)$ with a closed subset of $X!^S$ via the following embedding:
\begin{align*}
&\Hom(G,X!) \to \{ \alpha\in X!^{S}\mid \forall w\in R: w(\alpha) = 1_{X} \}\subset X!^{S}\\
&\rho\mapsto \{\rho(s_i)\}_{s_i \in S}
\end{align*}
Thus $\Hom(G,X!)$ is a closed subspace of $X!^S$ and the induced topology makes it into a Polish space (note that this topology does not depend on the choice of presentation). A basis for the topology on $\Hom(G,X!)$ is given by:
\begin{align*}
& \O(\rho,T,A) := \{ \sigma\in\Hom(G,X!)\mid  \forall t\in T: \sigma(t)\rest{A} = \rho(t)\rest{A} \} \\
& (\rho\in\Hom(G,X!),T\subset G \text{ and } A\subset X\text{ both finite}).
\end{align*}
If $S$ itself happens to be finite, then the sets $\O(\rho,S,A)$ form a basis for $\Hom(G,X!)$.

As mentioned in the introduction (see Equation \ref{eqn:action}) the group $X!$ acts, from the left, on $\Hom(G,X!)$ and the orbits of this action are exactly the standard isomorphism classes of permutation representations. It is well known that two permutation representations are isomorphic if and only if they contain the same transitive components, appearing with the same multiplicity. The transitive components, in turn, are isomorphic to quasiregular actions of the form $G \acts{} G/H$ for some $H \in \Sub(G)$. If $\{H_i\} \subset \Sub(G)$ is a countable or finite collection of subgroups and if $d_i \in \N \cup \{\infty\}$ we will denote the (isomorphism class of) the permutation representation that has exactly $d_i$ transitive components isomorphic to $G \acts{\eta_i} G/H_i$ by 
$$\bigsqcup_{i} d_i \cdot \left(G/H_i \right) = \bigsqcup_i d_i \cdot \eta_i.$$
Some care is due with this notation. It is not always possible to  identify such an action with an element of $\Hom(G,X!)$ because if the sum is finite and $H_i$ are all of finite index then the underlying set is finite. When this is not the case we can identify such an action with an element of $\Hom(G,X!)$ via an arbitrary bijection between $\sqcup_{i} G/H_i \cong X$. Different choices of this bijection will yield all the different points in the corresponding $X!$ orbit. 

\noindent We will make frequent use of the following:
\begin{definition}
Let $G\acts{} X, x\in X, g_1,\dots,g_n\in G$ and let $w= w_k w_{k-1}\cdots w_1$ be a word over $\{g_1^{\pm 1},\dots,g_n^{\pm 1}\}$. The \emph{trace} of $x$ under $w$ is the set:
\[ \trace{x}{w} = \{ x, w_1 x,\dots, w_{k-1}\cdots w_1 x, wx\}. \]
\end{definition}

\subsection{The space of subgroups}
Let $G$ be a countable group and consider the space $\{0,1\}^G$ of subsets of $G$, equipped with the product topology. This is a compact, metrizable space. Let $\Sub(G)$ denote the set of all subgroups of $G$. It is easy to verify that $\Sub(G)$ is closed in $\{0,1\}^G$ and so it is a compact, metrizable space. The induced topology on $\Sub(G)$ is called the Chabauty topology and a basis for this topology can be given by the sets
\[ \W(H,\Omega) = \{ K\in\Sub(G)\mid K\cap\Omega = H\cap\Omega \}\quad (H\in\Sub(G), \Omega \subset G\text{ finite}). \]
If $H \in \Sub(G)$ we denote by
\[ \Env(H) = \{K \in \Sub(G) \mid K \ge H\}\]
the \emph{envelope} of $H$. Both subsets $\Sub(H)$ and $\Env(H)$ are closed in $\Sub(G)$. If $H$ is finitely generated then $\Env(H)$ is also open. Denoting by $\Sub^{\fg}(G) \subset \Sub(G)$ the collection of finitely generated subgroups, it is easy to check that the collection 
\[ \{\Env(H) \mid H \in \Sub^{\fg}(G)\} \ \bigcup \ \{\Sub(G) \minus \Env(H) \mid H \in \Sub^{\fg}(G)\}\] 
forms another basis for the topology of $\Sub(G)$. 

\subsection{Isolated subgroups}
Let $\Is(G)$ and $\Occ(G)$ denote the isolated points of $\Sub(G)$ and the subgroups of $G$ with open conjugacy classes respectively. Note that a subgroup $H$ is in $\Occ(G)$ if and only if there is an open neighborhood of $H$, consisting of only conjugates of $H$. Both subsets are open and conjugation invariant.
\begin{proposition}\label{prop:is}
Here are some basic properties of these subgroups:
\begin{enumerate}
\item \label{itm:is_in_Occ} $\Is(G) = \Occ(G)$.
\item \label{itm:Occ_fg} Every $H \in \Is(G)$ is finitely generated.
\item \label{itm:env} if $H \in \Sub^{\fg}(G)$ then $\Env(H)$ is an open neighborhood of $H$. 
\item \label{itm:fi_Is} When $G$ itself is finitely generated then every finite index subgroup is isolated. 
\item \label{itm:fin_env} If $H \in \Sub^{\fg}(G)$ and $\left| \Env(H) \right| < \infty$ then $H$ is isolated. In particular every finitely generated maximal subgroup is isolated. 
\end{enumerate}
\end{proposition}
\begin{proof}
It is clear that $\Is(G) \subset \Occ(G)$ and that both sets are open. The opposite inclusion follows from Baire's theorem: Let $K \in \Occ(G)$ and let $[K] = \{g K g^{-1} \ | \ g \in G\}$ be its conjugacy class, which is open by definition. Since $G$ is countable, $[K]$ is a countable (or finite) union of closed points so by Baire's theorem one of them has to be open. Since $G$ is transitive on $[K]$ all of these points are open and in particular $K \in \Is(G)$. This proves (\ref{itm:is_in_Occ}). If $H$ is not finitely generated then we can find a sequence of finitely generated subgroups $H_1 < H_2 < H_3 < \ldots$ with $\cup_i H_i = H$. Clearly $H_i \to H$ in the topology of $\Sub(G)$ but none of these subgroups is equal to $H$ because $H$ is not finitely generated. (\ref{itm:env}) is clear and (\ref{itm:fin_env}) follows directly from (\ref{itm:env}). Finally if $G$ itself is finitely generated then so is every finite index subgroup and (\ref{itm:fi_Is}) follows directly from (\ref{itm:fin_env}).
\end{proof}
Thus, for a finitely generated group $G$ isolated subgroups form a class of subgroups that sits between the finitely generated subgroups and the subgroups of finite index, namely:
\[ \Sub^{\fg}(G) \subset  \Is(G) \subset \Sub^{\fg}(G). \]
We find it very useful to think of isolated subgroups as generalizations of finite index subgroups. 

\subsection{Generic properties}
A subset $A \subset Y$ in a Polish space is called {\it{generic}} (or alternatively {\it{residual}} or {\it{co-meager}}) if it contains a countable intersection of dense open sets. By Baire's category theorem generic sets are always dense. We say that {\it{the property (P) is generic in $Y$}} or that {\it{a generic element of $Y$ has the property (P)}} if the set $\{ y\in Y\mid y \text{ has the property (P)} \}$ is generic in $Y$. 

In this paper we will be interested in generic properties of permutation representations $\Hom(G,X!)$. The simplest example is $\Hom(\Z,X!) \cong X!$. The following, well known proposition, summarizes the generic properties of this space. Its proof is an exercise in Baire's category theorem which we leave to the readers. We chose to mention it here because our main theorem, and its proof, are basically far reaching generalizations of this fact. 
\begin{proposition}
$X!$ has a residual conjugacy class. This conjugacy class can be described explicitly: 
$$\tau^{*}=\bigsqcup_{n \in \N} \infty \cdot \left(\Z/n\Z\right)$$
\end{proposition}

In terms of the definition below, the above proposition just says that $G=\Z$ admits a generic permutation representation. 
\begin{definition}
We say that $G$ {\it{admits a generic permutation representation}} if there exists a permutation representation $\tau^{*} \in\Hom(G,X!)$ whose orbit under the action $X! \acts{} \Hom(G,X!)$ is residual in $\Hom(G,X!)$.
\end{definition}

\subsection{Properties of the stabilizer map}
Given a permutation representation $\sigma \in \Hom(G,X!)$ and a point $x \in X$, we denote by $G_x(\sigma) = \{g \in G \ | \ \sigma(g)x=x \}$ the stabilizer of this point. Fixing $x$ this gives rise to a {\it{stabilizer map}} 
\begin{eqnarray}
G_x: \Hom(G,X!)  & \arrow &  \Sub(G)\\
\nonumber \sigma & \mapsto &  G_x(\sigma)
\end{eqnarray}

\begin{lemma}\label{lem1} (Main lemma)
For every $x\in X$, the stabilizer map $G_x: \Hom(G,X!)  \to  \Sub(G)$ is continuous, surjective and open. 
\end{lemma}
\begin{proof}
It is clear that this map is surjective. Let $\phi\in\Hom(G,X!)$ and let $\Omega\subset G$ be finite. If $\sigma\in \mathcal{O}(\phi,\Omega,\{x\})$ then \[\forall g\in\Omega: g\in G_x(\sigma) \Leftrightarrow x = \sigma(g)x = \phi(g)x \Leftrightarrow g\in G_x(\phi)\] i.e $G_x(\sigma)\in\W(G_x(\phi),\Omega)$. This proves that the map $G_x$ is continuous. 

To prove that this map is open let $\tau \in \mathcal{O} = \mathcal{O}(\phi,S,A) \subset \Hom(G,X!)$ be a basic open neighborhood in $\Hom(G,X!)$ and a point therein. By extending $S$ we may assume that this set contains the identity and is symmetric, i.e. that $S = S^{-1}$. We have to exhibit an open neighborhood $\mathcal{W} \subset \Sub(G)$ such that $G_x(\tau) \in \mathcal{W} \subset G_x(\mathcal{O}(\phi,S,A))$. 

Let $Y = \tau(G)x \subset X$ be the orbit of $x$ under $\tau$ and $L =  G_x(\tau)$ be the stabilizer. We can identify $Y = G/L$ under the orbit map $gL \mapsto \tau(g)x$.  Set $A_Y := A \cap Y$.   Let $\Omega \subset G$ be a finite symmetric set such that
\[ \bigcup_{s \in S} \tau(s)\overline{A_Y} \subset \tau(\Omega)x = \{\tau(\omega)x \mid \omega \in \Omega\}.\]
We claim that the basic open set $\mathcal{W} = \mathcal{W}(L,\Omega S \Omega)$ satisfies all our requirements. Where $\Omega S \Omega = \{\omega_1 s \omega_2 \ | \ \omega_i \in \Omega, s \in S \}$ is the set product. Fixing a group $K \in \mathcal{W}$, we will complete our proof by finding $\eta \in \mathcal{O}$ such that $K = G_x(\eta)$. 

Consider the finite subset $\Omega K = \{\omega K \in G/K \ | \ \omega \in \Omega\}$ and define a partial map:
\begin{eqnarray*}
f: \Omega K & \rightarrow &  Y \\
\omega K  & \mapsto &  \tau(\omega)x, \qquad \forall \omega \in \Omega.
\end{eqnarray*}
This map is well-defined and injective on its domain, since by the definition of the open set $\mathcal{W}$ we have $\omega_1 K = \omega_2 K \iff \omega_1 L_1 = \omega_2 L_1 \iff  \omega_1 L = \omega_2 L \iff \tau(\omega_1)x = \tau(\omega_2)x$ for every $\omega_1,\omega_2 \in \Omega$. For a similar reason this map partially respects the $G$ action on both sides in the sense that 
$$s f(\omega K) = f(s \omega K) \quad \forall \omega \in \Omega, s \in S.$$ 
Define a bijection 
$\tilde{f}: G/K \sqcup X \rightarrow X$ satisfying the following conditions:
\begin{itemize}
\item $\tilde{f}$ extends $f$, namely $\tilde{f}(\omega K) = f (\omega K), \quad \forall \omega \in \Omega$.
\item $\tilde{f}$ is the identity on $\cup_{s \in S} \tau(s)(A) \setminus Y$. 
\end{itemize}
Now let us define an action $\sigma_1$ of $G$ on $G/K \sqcup X$ where $G$ acts on $G/K$ by the quasi-regular action on on $X$ by $\tau$ and let $\sigma = \in \Hom(G,X!)$ be defined by $\sigma(g)x = \tilde{f} \sigma_1(g) \tilde{f}^{-1} (x)$. 
It is easy to verify that $K = G_x(\sigma)$. Also for every $s\in S$ and $y \in A$ we have $\sigma(s)y = \tau(s)y$. Indeed if $y \in A_Y$ then by our choice of $\Omega$, we have $y = \tau{\omega} x $ for some $\omega \in \Omega$ and hence:
\begin{equation*}
\sigma(s) y =  \left\{ \begin{array}{ll}
\widetilde{f} \sigma_1(s) \widetilde{f}^{-1} y = f s f^{-1} \tau(\omega) x = f s \omega K = \tau(s \omega) x = \tau(s) y, & y \in A_Y \\
\widetilde{f} \sigma_1(s) \widetilde{f}^{-1} y = \widetilde{f} \tau(s) \widetilde{f}^{-1} y = \tau(s)y & y \not \in A_Y
\end{array} \right.
\end{equation*}
This concludes the proof that $G_x$ is open. 
\end{proof}

\begin{corollary} \label{cor:dense_dense}
For a subset $\mathcal{D} \subset \Sub(G)$ denote\[\widetilde{\mathcal{D}} := \{\phi \in \Hom(G,X!) \mid G_x(\phi) \in \mathcal{D} \ \forall x \} = \bigcap_{x \in X} (G_x)^{-1}(\mathcal{D}).\]
Then
\begin{enumerate}
\item If $\mathcal{D}$ is conjugation-invariant and dense in $\Sub(G)$ then $\widetilde{\mathcal{D}}$ is dense in $\Hom(G,X!)$.
\item If $\mathcal{D}$ is $G_{\delta}$ in $\Sub(G)$ then $\widetilde{\mathcal{D}}$ is $G_{\delta}$ in $\Hom(G,X!)$.
\end{enumerate}
In particular, $\widetilde{\mathcal{D}}$ is generic whenever $\mathcal{D}$ is. 
\end{corollary}
\begin{proof}
Suppose $\mathcal{D}$ is conjugation-invariant and dense in $\Sub(G)$, let $\rho\in\Hom(G,X!)$ and let $S\subset G,\ A\subset X$ be finite. Let $x\in A$. By Lemma \ref{lem1}, the set $G_x^{-1}(\mathcal{D})$ is dense in $\Hom(G,X!)$ and so there exists $\phi_1\in\Hom(G,X!)$ such that $\phi_1\in\O(\rho,S,A)$ and $G_x(\phi_1)\in\mathcal{D}$. Denote $Y_1 = \phi_1(G)x$ and note that since $\mathcal{D}$ is conjugation-invariant we have $G_y(\phi_1)\in\mathcal{D}$ for all $y\in Y_1$. Now, if $z\in A\minus Y_1$, we can apply the same argument and get a permutation representation $\phi_2$ of $G$ on $X\minus Y_1$ such that $\phi_2(s)$ agrees with $\phi_1(s)$ on $A\cap (X\minus Y_1)$ for all $s\in S$ and such that all the stabilizers of points belonging to $Y_2 = \phi_2(G)z$ are in $\mathcal{D}$. We get that the action $\phi\in\Hom(G,X!)$ defined by
\[ \forall g\in G, x\in X: \phi(g)x =
\begin{cases}
\phi_1(g)x,& x\in Y_1\\
\phi_2(g)x,& x\in X\minus Y_1
\end{cases}\]
belongs to $\O(\rho,S,A)$ and every stabilizer of a point belonging to the $\phi$-invariant set $Y_1\cup Y_2$ is in $\mathcal{D}$. By repeating the process described above, we get after finitely many steps an action $\psi\in \O(\rho,S,A)$ and a $\psi$-invariant set $Y\subset X$ such that $A\subset Y$ and such that $G_y(\psi)\in\mathcal{D}$ for every $y\in Y$. Finally, extend the action $G\acts{\psi} Y$ to an action $\widetilde{\psi}\in\Hom(G,X!)$ in such a way that all stabilizers of points in $X\minus Y$ belong to $\mathcal{D}$. Thus, $\widetilde{\psi}\in\widetilde{\mathcal{D}}$ and since $A\subset Y$ we have $\widetilde{\psi}(s)\rest{A} = \psi(s)\rest{A} = \rho(s)\rest{A}$ for all $s\in S$, i.e. $\widetilde{\psi}\in\O(\rho,S,A)$. This proves part 1.

Now, assume $\mathcal{D}$ is $G_\delta$ so we can write $\mathcal{D} = \bigcap_n \mathcal{D}_n$, where $\mathcal{D}_n$ is open for every $n\in\N$. Then, for every $x\in X$ we have $G_x^{-1}(\mathcal{D}) = \bigcap_n G_x^{-1}(\mathcal{D}_n)$. Since $\mathcal{D}_n$ is open we get from Lemma \ref{lem1} that $G_x^{-1}(\mathcal{D}_n)$ is open for every $x\in X$ and $n\in\N$ and so $G_x^{-1}(\mathcal{D})$ is $G_\delta$. Since $X$ is countable this means that $\widetilde{\mathcal{D}} =  \bigcap_{x \in X} (G_x)^{-1}(\mathcal{D})$ is also $G_\delta$ and part 2 is proven.
\end{proof}

\begin{lemma}\label{lem2}
Let $U \subset \Sub(G)$ be open and non-empty. Then for a generic permutation representation $\phi \in \Hom(G,X!)$ we have 
\[ |\{x \in X\mid G_x(\phi) \in U \}| = \infty.\]
\end{lemma}
\begin{proof}
Let
\begin{align*}
&\Lambda_m  =  \left\{\phi \in \Hom(G,X!) \ | \ |\{x \in X\mid G_x(\phi) \in U \}| > m \right\}. \\
&\Lambda =  \bigcap_{m \in \N} \Lambda_m.
\end{align*}
We claim that $\Lambda$ is generic, and by Baire's theorem it would be enough to show that $\Lambda_m$ is open and dense for each $m \in \N$. The fact that $\Lambda_m$ is open follows directly from the continuity of the map $G_x$. To prove density we just add $m$ new orbits with a stabilizer from $U$ far away. This is made explicit in the following way. Given a basic open set $\mathcal{O}(\phi,S,A) \subset \Hom(G,X!)$ we want to find an element $\eta \in \mathcal{O}(\phi,S,A) \cap \Lambda_m$. Fix any subgroup $H \in U$ and consider the set
$Y = X \sqcup \left(G/H \right)^m$ endowed with the diagonal $G$ action $\eta' \in \Hom(G,Y!)$ given by $\eta'(g)(x,g_1H, \ldots, g_mH) = (\phi(g)x, gg_1H, \ldots, gg_mH)$. Now, let $\iota: X \rightarrow Y$ be the identity map from $X$ to the copy of $X$ contained in $Y$, $\iota|_{\overline{A}}: \overline{A} \rightarrow Y$ be its restriction to the finite set $\overline{A} =A \cup \left(\bigcup_{s \in S \cup S^{-1}} \phi(s)A\right)$ and let $I: X \rightarrow Y$ be an extension of $\iota|_{\overline{A}}$ to a bijection between $X$ and $Y$. One easily checks that \[\eta = I^{-1} \circ \eta' \circ I \in \mathcal{O}(\phi,S,A) \cap \Lambda_m\] is as required. 
\end{proof}

\begin{lemma}\label{lem2bis}
To any given $\rho \in \Hom(G,X!)$ there is an arbitrarily close action, with infinitely many fixed points.
\end{lemma}
\begin{proof}
Given finite sets $S \subset G, A \subset X$ we seek an action $\rho' \in \mathcal{O}(\rho,S,A)$ with infinitely many fixed points. Consider the action $G \acts{(\rho, \operatorname{1})} X \sqcup \N$, obtained from $\rho$ by adding countably many fixed points. The desired action $\rho' = \phi^{-1} (\rho,\operatorname{1}) \phi$ is obtained by intertwining this action via any bijection $\phi: X \to X \sqcup \N$ with the property that $\phi$ is the identity when restricted to $A\cup \big(\bigcup_{s\in S} \rho(s)A\big)$.
\end{proof}
%

\section{The LERF property}\label{sec:LERF}
In this section we prove Theorem \ref{thm:LERF}. In fact, as promised in the introduction, we will prove the following slightly more general version of the theorem that for arbitrary countable groups (not necessarily finitely generated). 
\begin{theorem}\label{thm:LERF_countable}
Let $G$ be a finitely generated group, then the following conditions are equivalent:
\begin{enumerate}
\item $G$ is LERF.
\item The collection of finite index subgroups of $G$ is dense in $\Sub(G)$.
\item \label{itm:fin_org_dense} The collection of permutation representations all of whose orbits are finite is dense in $\Hom(G,X!)$. 
\item  \label{itm:fin_orb_fg} The collection of permutation representations all of whose orbits are finite is generic in $\Hom(G,X!)$.
\item \label{itm: typical_action} $G$ has a generic permutation representation, all of whose orbits are finite. 
\end{enumerate}
When $G$ is countable, but not necessarily finitely generated, then only the first three conditions are equivalent. 
\end{theorem}
It is easy to verify that this theorem implies Theorem \ref{thm:LERF} when $G$ is finitely generated. 
The apparent complications when $G$ fails to be finitely generated, and lack thereof in the proof of Theorem \ref{thm:main}, 
emphasize one of our main points: {\it{Isolated subgroups are more natural than finite index subgroups in this setting.}} We mention again that the equivalence of the first, third and fourth conditions is already present 
in Rosendal's works \cite{Rosendal1} and \cite{Rosendal2} (see the discussion in the introduction).

\begin{proof}[Proof of Theorem \ref{thm:LERF_countable}]
In a countable group $G$ every subgroup is an ascending union of finitely generated subgroups. Hence $\Sub^{\fg}$ is dense in $\Sub(G)$. The LERF property implies that $\Sub^{\fg} \subset \overline{\Sub^{\fii}}$ as every finitely generated subgroup is a descending intersection of finite index subgroups. This shows {\textit{(1) $\implies$ (2)}}. Now  {\textit{(2) $\implies$ (3)}} follows directly from Corollary (\ref{cor:dense_dense}). To prove {\textit{(3) $\implies$ (1)}}, assume we are given a finitely generated infinite index subgroup $L = \la S\ra$  and $g \in G \minus L$. Let $\phi \in \Hom(G,X!)$ be any permutation representation which is isomorphic to the quasiregular action $G \acts{} (G/L)$ in such a way that $x \in X$ is identified with the trivial coset $eL$. If $\psi \in \O(\phi,S \cup \{g\},\{x\})$ is an action with finite orbits then $[L:G_x(\psi)] < \infty$, $L < G_x(\psi)$ but $g \not \in G_x(\psi)$ proving the LERF property. 

Assume now that $G$ is finitely generated. The implications \emph{(5)$\implies$(4)$\implies$(3)} are obvious, so it is enough to prove the implication \emph{(2)$\implies$(5)}. We start by describing the generic permutation representation $\tau^* \in \Hom(G,X!)$. Let $\Sub^{\fg}(G) = \{ H_1,H_2,\dots\}$ be an enumeration of the finite index subgroups of $G$, let $\phi_n$ be the quasi-regular representation of $G$ on $G/H_n$ and define:
\[ \tau^* =  \bigsqcup \infty \cdot \phi_n.\]
Namely, take countably many copies of each representation in the list, and let $G$ act naturally on the disjoint union of the corresponding sets. It follows from Corollary \ref{cor:dense_dense}, applied to the open dense set $\mathcal{D} = \Sub^{\fii}(G) \subset \Sub(G)$, that the collection of permutation representations all of whose orbits are finite is generic. It follows from Lemma \ref{lem2} applied to the open set $\{H_i\} \subset \Sub(G)$ that the collection of permutation representations in which $\rho_i$ appears countably many times as a transitive component is also generic. By Baire's category theorem a generic permutation representation has only finite 
orbits and each $\rho_i$ appears in it countably many times. But such a permutation representation must be permutation isomorphic to $\tau^{*}$. 
\end{proof}
\begin{remark}\label{remark1}
As the examples below demonstrate if $G$ is an infinitely generated LERF group it is no longer true that a generic permutation representation has only finite orbits. It is still true, however, that the restriction of such a generic action to every finitely generated subgroup $H < G$ has only finite orbits. It is even true that the restriction of a generic permutation to $H$ admits a well defined isomorphism type (up to isomorphism of permutation representations of $H$). The details of the proof are quite similar to our proof above and we leave them to the reader. \end{remark}
In order to demonstrate the use of Theorem \ref{thm:LERF} above and give some basic examples we analyze the situation in free groups. Providing a short proof to Hall's theorem that free groups are LERF. 
\begin{proposition}
Let $F_n, \ 1 \le n \le \infty$ be a free group. 
\begin{enumerate}
\item \label{itm:Fn_LERF} $F_n$ is LERF. 
\item \label{itm:transitive} A generic permutation representation $\phi \in \Hom(F_{\infty},X!)$ is transitive. 
\item \label{itm:non_fg} If $\Gamma$ is a countable LERF group that is not finitely generated then $\Sub(\Gamma)$ is perfect.\end{enumerate}
\end{proposition}
\begin{proof}
First notice that $\Hom(F_n,X!) = X!^n$. Let $X!^{(f)} < X!$ be the dense subgroup of finitely supported permutations. Clearly $\Hom(F_n,X!^{(f)}) = \left(X!^{(f)} \right)^n \subset X!^n = \Hom(F_n,X!)$ is a dense set of permutation representations all of whose orbits are finite. This proves (\ref{itm:Fn_LERF}) by establishing Theorem \ref{thm:LERF}(\ref{itm:fin_org_dense}). 

To prove (\ref{itm:transitive}) it is enough, by Baire's theorem, to show that the set $\Theta(x,y) = \{ \phi \in \Hom(F_{\infty},X!) \ | \ y \in \phi(F_{\infty})x \}$ is open and dense. Openness is obvious. For the density, fix a free generating set $F_{\infty} = \langle x_1,x_2,\ldots \rangle$. Given a basic open set $\mathcal{O}(\phi, S,A) \subset \Hom(F_{\infty},X!)$ the finite set $S \subset F_{\infty}$ contains words that involve only finitely many of the generators, say $S \subset \langle x_1, x_2, \ldots, x_r \rangle$. We can find $\sigma \in \mathcal{O}(\phi,S,A) \cap \Theta(x,y)$ by setting $\sigma(x_i) = \phi(x_i)$ for every $1 \le i \le r$ and then defining $\sigma(x_{r+1})$ in such a way that $\sigma(x_{r+1})x = y$. 

Finally if $\Gamma$ fails to be finitely generated then so do her finite index subgroups. So by Proposition \ref{prop:is}(\ref{itm:Occ_fg}) none of these are isolated. If $\Gamma$ is also LERF then the finite index subgroups are dense and in particular there can be no isolated subgroups at all. This proves (\ref{itm:non_fg}).
\end{proof}

\section{Solitary groups}\label{sec:genact}

This section is dedicated to the proof of Theorem \ref{thm:main}. 
\begin{proof}
Assume that $\Is(G)$ is dense in $\Sub(G)$. Let $\Is(G)= \{\Delta_1, \Delta_2, \Delta_3, \ldots, \Sigma_1, \Sigma_2, \ldots \}$ be representatives for the conjugacy classes of these isolated subgroups of $G$. Where we made a distinction between the groups $\Delta_i$ that are of finite index in their normalizer and the groups $\Sigma_i$ that are not. We denote by $G \stackrel{\delta_i}{\curvearrowright} G/\Delta_i, G \stackrel{\sigma_i}{\curvearrowright} G/\Sigma_i$ the corresponding quasiregular actions. It is important to note that in $G/\Delta_i$ there are finitely many points whose stabilizer is $\Delta_i$ under the $\delta_i$ action. In $G/\Sigma_i$ there are infinitely many similar points. With this terminology in place we can describe the generic permutation representation. It will have countably many orbits isomorphic to each $\delta_i$ and one orbit isomorphic to each $\sigma_i$:
$$\tau^{*} =  \left(\bigsqcup_i \aleph_0 \delta_i \right) \sqcup \left(\bigsqcup_j \sigma_j \right).$$

Applying Corollary \ref{cor:dense_dense} to the open dense subset $\Is(G) \subset \Sub(G)$ we conclude that a generic permutation representation has all of its stabilizers in $\Is(G)$. In other words a generic  $\left( \sqcup_i d_i \delta_i \right) \sqcup \left(\sqcup_i s_i \sigma_i \right)$ for some $d_i,s_i \in \N \cup \{\infty\}$. By Lemma \ref{lem2}, applied to the open set $\{\Delta_i\}$ we know that a generic representation has infinitely many points whose stabilizer is $\Delta_i$, which immediately implies that $d_i = \infty, \ \forall i$. Note that an identical argument tells us that a generic representation has countably many points whose stabilizer is $\Sigma_i$, but even one orbit isomorphic to $\sigma_i$ is enough to ensure that so this does not add any information about the coefficients $s_i$. 

All that is left to prove is that a generic representation has only one orbit isomorphic to $\sigma_i$ for every $i$. In order to simplify the notation we will hence fix the index $i$ and denote $\sigma = \sigma_i$, $\Sigma = \Sigma_i$. Let us denote by $[\Sigma] = \{g \Sigma g^{-1} \ | \ g \in G \} \subset \Sub(G)$ the conjugacy class of $\Sigma$. The bad event is the existence of two different orbits with stabilizers 
in $[\Sigma]$:
\begin{eqnarray*}
\Theta & = &  \bigcup_{x,y \in X} \Theta_{x,y} \\
\Theta_{x,y} & = &  \{\phi \in \Hom(G,X!) \ | \ G_x(\phi) = G_y(\phi) = \Sigma; {\text{ but }} \phi(G)x \ne \phi(G)y \} \\
 \end{eqnarray*}
By continuity of the stabilizer map (Lemma \ref{lem1}) the sets $\G_x^{-1}(\{\Sigma\}), G_y^{-1}(\{\Sigma\})$ and hence also $\Theta_{x,y}$ are closed. So by  Baire's theorem it suffices to prove that $\Theta_{x,y}$ is nowhere dense. Assume to the contrary that $\mathcal{O} = \mathcal{O}(\phi,S,A) \subset \Theta_{x,y}$ for some basic open set. Replacing if necessary $\mathcal{O}$ by a smaller basic open set, we may assume that $S = S^{-1}$ and that $O \subset G_x^{-1}(\{\Sigma\}) \cap G_y^{-1}(\{\Sigma\})$. Let $A_x = \phi(G)x \cap A$ and $\Omega_x = \bigcup_{s \in S} \phi(s)A_x$. We define $A_y,\Omega_y$ similarly. 

Consider the quasiregular action $\sigma: G \curvearrowright G/\Sigma$, if $g \in N_G(\Sigma)$ then there is a unique $G$ invariant isomorphism 
\begin{eqnarray*}
\eta_{g\Sigma}: \phi(G)x  & \to & G/\Sigma \\
\phi(h) x & \mapsto & hg\Sigma.  
\end{eqnarray*}  
Since $[N_G(\Sigma):\Sigma] = \infty$ by assumption there are infinitely many possible choices of  points $g \Sigma$ that would work for $x$ and similarly for $y$. Let $g_x \Sigma, g_y \Sigma$ be two such choices satisfying the additional property 
$$\eta_{g_x}(\Omega_x) \cap \eta_{g_y}(\Omega_y) = \emptyset.$$

Let $\alpha: \phi(G)x \sqcup \phi(G)y \to G/\Sigma$ be any bijection such that $\alpha(z) = \eta_{g_x}(z), \forall z \in \Omega_x$ and $\alpha(z) = \eta_{g_y}(z), \forall z \in \Omega_y$. We define a new action $\psi \in \mathcal{O}$ by the following formula:
$$
\psi(g)(x) = \left \{ 
\begin{array}{ll}
\alpha^{-1}(g \alpha(x)) & {\text{if }} x \in \phi(G)x \sqcup \phi(G) y \\
\phi(g)x & {\text{otherwise}} 
\end{array} \right. 
$$
It is easy to verify that $\psi \in \mathcal{O}$, that $G_x(\psi),G_y(\psi) = \Sigma$ and that $x,y$ are in the same $\psi(G)$-orbit. Which completes the proof of the first implication. 

Assume now that there exists a generic permutation representation $\tau^* \in \Hom(G,X!)$, by assumption its isomorphism class $\Phi = \{a \tau^* a^{-1} \ | \ a \in X!\}$ is co-meager in $\Hom(G,X!)$. The collection of subgroups appearing as point stabilizers of $\tau^*$ are given by 
$O = \{G_x(\tau^*) \ | \ x \in X\} = G_{x_0}(\Phi) \subset \Sub(G),$ 
where $x_0$ is an arbitrary basepoint. Since, by Lemma \ref{lem1}, the map $G_{x_0}$ is surjective and continuous $O$ is dense in $\Sub(G)$. In particular $O \supset \Is(G)$. 

We will show that $O \subset \Is(G)$ thus showing that the latter is dense and completing the proof.  Let $\Sigma \in O$  and $[\Sigma] = \{g \Sigma g^{-1} \ | \ g \in G\}$ its conjugacy class. By Proposition \ref{prop:is}(1) it is enough to show that $[\Sigma]$ is open. If $[\Sigma]$ fails to be open it must have an empty interior, because $G$ acts transitively on $[\Sigma]$. Since $[\Sigma]$ is countable it follows from Baire's theorem that $\Sub(G) \setminus [\Sigma]$ is a dense $G_{\delta}$ set. By Corollary \ref{cor:dense_dense} 
$$\widetilde{\Sub(G) \setminus [\Sigma]} = \{\sigma \in \Hom(G,X!) \ | \ G_x(\sigma) \not \in [\Sigma], \ \forall x \in X\}$$ is also a dense $G_\delta$ set. But this contradicts the fact that $\Phi$ is dense $G_{\delta}$ as the intersection of these two sets is empty. 
\end{proof}

\noindent We conclude this section by proving Theorem \ref{thm:solitary_prop}.
\begin{proof} (of Theorem \ref{thm:solitary_prop})
In view of the fact that in a finitely generated group $G$ finite index subgroups are isolated (\ref{itm:l->s}) follows directly from the comparison of Theorems \ref{thm:main} and \ref{thm:LERF}. Density of isolated points (statement \ref{itm:countable}) is a general fact about countable Baire spaces. Indeed the set $\Sub(G) \setminus \Is(G)$ is nowhere dense since it is a countable union of closed, nowhere dense points. Consider a short exact sequence as in statement (\ref{itm:fg_kernel}). Since $N$ is finitely generated $\Env(N)$ is clopen. It is easy to verify that the correspondence principle, between subgroups of $G$ and subgroups of $H$ containing $N$, gives rise to a homeomorphism 
\begin{eqnarray*}
\Env(N) & \arrow & \Sub(G) \\
\Theta & \mapsto & \phi(\Theta)
\end{eqnarray*}
Claim (\ref{itm:fg_kernel}) follows immediately. 

It is well known that the free product of two LERF groups is LERF \cite{rom:lerf_free_prod,burns:lerf_free_prod}. If one of the groups, say $H$ is trivial then $G*\langle e \rangle \cong G$ and the situation is clear. Thus to establish (\ref{itm:free_prod}) we have to show that if neither group is trivial and $G$ fails to be LERF then $G*H$ cannot be solitary. Let $\Sigma \in \Sub^{\fg}(G) \setminus \overline{\Sub^{\fii}}$ be some finitely generated subgroup that cannot be approximated by finite index subgroups. Let $\eta: G*H \arrow G$ be the map that is the identity on $G$ and trivial on $H$. We will show that $\Xi := \eta^{-1}(\Sigma) = \Sigma \ker(\eta) \in \Sub(G*H)$ is a subgroup that cannot be approximated by isolated subgroups. Indeed let $\Omega \subset G$ be some finite set such that the neighborhood $\W_{\Sub(G)}(\Sigma, \Omega)$ does not contain any finite index subgroup. We can use the same $\Omega$ to define an open neighborhood in $G*H$ and it is clear that $[G: \Delta \cap G] = \infty$ for every $\Delta \in \Sub_{\Sub(G*H)}(\Xi, \Omega)$. Thus our theorem is proved in view of the following lemma, which seems very useful in its own right. \end{proof}
\begin{lemma}
Let $G,H$ be two countable groups with $G$ infinite and $H$ non-trivial. If $\Delta \in \Is(G*H)$ then $[G: G \cap \Delta] < \infty$.
\end{lemma} 
\begin{proof}
It will be more convenient to argue at the level of actions and Schreier graphs. Note that $\Hom(G*H,X!) \cong \Hom(G,X!) \times \Hom(H,X!)$. We will denote this isomorphism by $\phi*\psi \mapsto (\phi,\psi)$, namely $\phi*\psi$ is the unique action of $G*H$ whose restriction to $G$ is $\phi$ and to $H$ is $\psi$.

Let $\Delta \in \Sub(G*H)$ and assume that $[G:G\cap\Delta] = \infty$. Let $\phi*\psi \in \Hom(G*H,X!)$ be any action which is isomorphic to the quasiregular action $G*H \acts{} (G*H)/\Delta$ with $x \in X$ identified with the trivial coset $e \Delta$. Note that, while the action of $G*H$ is transitive, $\phi,\psi$ themselves need not be transitive. Still, by our assumption we know that the orbit $Y:=\phi(G)x \subset X$ is infinite. The argument now is simple enough : we obtain approximating actions of the form $\phi*\psi_n \stackrel{n}{\arrow} \phi*\psi$ by carrying out small perturbations on the action of $H$. Since the orbit $Y$ is infinite we can do arbitrarily small such perturbations on the action while still affecting the stabilizer of the point $x$. We elaborate below but this is basically a complete proof. 

 Let $x \in A_0 \subset A_1 \subset A_2 \ldots$ be finite sets ascending to the whole of $X$, $T_n \subset H$ finite sets ascending to a generating set of $H$. We can assume that $1 \in T_n = T_n^{-1}$. Of course if $H$ is finitely generated we can just take $T_n = T, \ \forall n \in \N$ to be some fixed symmetric set of generators. Set $B_n = \cup_{t \in T_n} \psi(t)(A_n)$ and $\xi_n \in Y \setminus B_n$. Since $X$ is infinite we can find bijections $f_n: X \arrow X \times \{0,1\}$, with the additional properties that
$f_n(x) = (x,0), \forall x \in A_n$ and $f_n(\xi_n) = (\xi_n,1)$.
Let $\eta \in \hom(H,X!)$ be any fixed action of $H$ on $X$. Using all this data we construct a sequence of actions $\overline{\psi^{\eta}} \in \Hom(H,(X \times \{0,1\})!)$ as follows:
$$\overline{\psi^{\eta}} (h)(x,i) = 
\left\{ \begin{array}{ll} 
            (\psi(h)x,0), & i = 0 \\
            (\eta(h)x,1), & i = 1 
         \end{array} \right.  $$
and let $\psi^{\eta}_n(h) = (f_n)^{-1} \circ \overline{\psi^{\eta}} \circ f_n$. It is clear from the definitions that $\psi^{\eta}_n \in 
\O(\psi,T_n,A_n)$ and in particular $\psi^{\eta}_n \arrow \psi$ as $n \arrow \infty$. Consequently of course $\phi*\Psi^{\eta}_n \arrow \phi * \psi$ in $\Hom(G*H, X!)$. 

Let $1,\lambda \in \Hom(H,X!)$ be the trivial and the regular left action of $H$ on $X$. The latter is defined via an arbitrarily chosen identification of $X$ with $H$ which will not play a role in the discussion. If $H$ is finite we replace the regular left action by countably many copies of the same action, just to make sure $\lambda$ is an action on an infinite set.  We use these to obtain two convergent sequences of actions and hence two convergent sequences of subgroups
\begin{align*}
\phi*\psi^{1}_n & \arrow  \phi*\psi & \qquad  \phi* \psi^{\lambda}_n & \arrow  \phi*\psi \\
(G*H)_x(\phi*\psi^{1}_n) & \arrow   \Delta & \qquad (G*H)_x(\phi* \psi^{\lambda}_n) & \arrow  \Delta
\end{align*} 
These sequences are different because if $g_n \in G_n $ is any element such that $\phi(g_n)x = \xi_n$ then for any $1 \ne h \in H$ we have  
$$g_n^{-1} h g_n \in (G*H)_x(\phi*\psi^{1}_n) \setminus (G*H)_x(\phi* \psi^{\lambda}_n).$$
Thus at least one of these sequences is not eventually constant, proving that the limit point $\Delta$ is not isolated. 

 \end{proof}

 %
%
%

  \section{Sketch of another proof of Theorem \ref{thm:main}}
We briefly sketch another proof of Theorem \ref{thm:main}, which is more along the lines of the arguments in \cite{Kechris_Rosendal}. First, we note that for any countable group $G$, the action of $X!$ on $Hom(G,X!)$ is topologically transitive, that is, given any two nonempty open subsets $U,V$ of $\Hom(G,X!)$, there always exists $\alpha \in X!$ such that $\alpha \cdot U \cap V \ne \emptyset$ (equivalently, there exists elements in $\Hom(G,X!)$ which have a dense conjugacy class). This is true simply because any two actions $\pi_1, \pi_2$ of $G$ on an infinite countable set embed in a third one $\pi_3$ (for instance, obtained by considering a disjoint union of two infinite countable sets, with $G$ acting as $\pi_1$ on the first copy, and $\pi_2$ on the second copy). Then the closure of the conjugacy class of $\pi_3$ contains both $\pi_1$ and $\pi_2$, proving the desired result.

This brings us to the setting of the following lemma; the equivalence between \eqref{l:comeager} and \eqref{l:Christian} below is the criterion we will be using, and is due to C. Rosendal. The equivalence of these conditions with \eqref{l:Yair} appears to be new, and seems potentially useful so we are including it here even though it will not be needed.

\begin{lemma}\label{l:WAP} 
Assume that $P$ is a Polish group acting continuously and topologically transitively on a Polish space $Z$. Then the following conditions are equivalent:
\begin{enumerate}
\item \label{l:comeager} There exists a comeager orbit. 
\item \label{l:Yair} For any open identity neighborhood $1 \in V \subset P$ the collection of points 
$$\{z \in Z \ | \ V(z) {\text{ is somewhere dense}}\} = \{z \in Z \ | \ \mathrm{Int}(\overline{V(z)}) \ne \emptyset \}$$
is dense in $Z$. 
\item \label{l:Christian} For any open identity neighborhood $1 \in V \subset P$ and any nonempty open subset $U$ of $Z$, there exists a nonempty open $U' \subseteq U$ such that, for any nonempty open $W_1,W_2 \subseteq U'$, one has $V W_1 \cap W_2 \ne \emptyset$. 
\end{enumerate}
\end{lemma}

\begin{proof}
Fix an open identity neighborhood $V$ and a sequence of group elements such that $\cup_n \alpha_n V = P$. If there exists a comeager orbit then by Baire's category theorem for every $z$ in this orbit there exists an $n$ such that $\alpha_nV(z)$ is somewhere dense. Translating by $\alpha_n$ we deduce that $V(z)$ itself is somewhere dense. This proves that \eqref{l:comeager} implies \eqref{l:Yair}.

Assume that \eqref{l:Yair} holds; fix an open identity neighborhood $1 \in V \subset P$ and a nonempty open subset $U$ of $Z$. Using our assumption, and continuity of the action, we may find a symmetric open identity neighborhood $V'$ such that $V'V' \subseteq V$, an open $U' \subseteq U$ and $x \in U$ such that the closure of $V'(x)$ contains $U'$. Then, for any $W_1,W_2$ nonempty open and contained in $U'$, we have some $v_1,v_2 \in V'$ such that $v_1 x \in W_1$, $v_2 x \in W_2$. Thus $v_2 v_1^{-1}W_1 \cap W_2 \ne \emptyset$ and \eqref{l:Christian} holds.

Finally, assume that \eqref{l:comeager} is false; since there exist dense orbits, any orbit must be meager or comeager, by the $0$--$1$ topological law \cite{Rao:top_0_1}. So in this case all orbits are meager. Given $z \in Z$, we then have an family of closed subsets $F_n$ with empty interior such that $P(z) \subseteq \bigcup F_n$; some $\{g \in P \colon g z \in F_n\}$ must have nonempty interior, proving that there exists some open neighborhood $V$ of $1$ such that $V(z)$ is nowhere dense.
Thus $Z$ is the union of sets of the form $\{z \in Z \colon V(z) \text{ is nowhere dense}\}$, where $V$ ranges over a countable basis of neighborhoods of $1$; one of these sets must be nonmeagre, hence (since these sets are Borel) comeagre in some nonempty open $U$. Assume that \eqref{l:Christian} holds, and pick $U' \subseteq U$ witnessing it. The assumption of \eqref{l:Christian} amounts to saying that $\{z \in U' \colon V(z) \cap W \ne \emptyset\}$ is dense open in $U'$ for any nonempty open $W \subseteq U'$; this implies that $\{z \in U' \colon V(z) \text{ is dense in } U'\}$ is comeagre in $U'$. This is a contradiction with the fact that $V(z)$ must be nowhere dense for a generic element of $U$, hence also of $U'$.

\end{proof}

Now, we need to understand when the above criterion is satisfied, for $G$ a countable group, $P=X!$ and $Z= \Hom(G,X!)$. Given an open set $U=\O(\rho,S,A)$, let $i(\rho)$ denote the number of distinct
$G$-orbits of elements of $A$; we may pick $\rho$ such that $i(\rho)$ number is minimal among elements of $U$. Then, enlarging $S$ and shrinking $A$ as needed, we can reduce to the situation where
$$\forall \pi \in U \ \forall a \ne a' \in A \ \forall g \in G \  \pi(g)(a) \ne a'\ .  $$
 We are now in a situation where orbits of elements of $A$ cannot interfere with each other; this enables us to reduce to the case where $A=\{a\}$ is a singleton and we are working inside the Polish space $Z'$ of \emph{transitive} $G$-actions. Consider an open set $U=\O(\rho,S,a) \cap Z'$; let $G_a$ denote the stabilizer of $a$ for this action, and let $V$ be the group of permutations fixing a finite set $F$. Enlarging $S$ if needed, we assume that $F \subseteq Sa$. Then, it is readily checked that two elements $\rho_1,\rho_2$ of $U$ belong to the same $V$-orbit iff the stabilizers of $a$ for $\rho_1$ and $\rho_2$ are the same. 
Let 
$$S_1= \{g_2 g_1^{-1} \colon g_1,g_2 \in S \text{ and } g_1(a)=g_2(a) \} \  , \ S_2= \{g \in S \colon g(a) \neq a\} \ .$$
The discussion above shows that the criterion \eqref{l:Christian} of Lemma \ref{l:WAP} is satisfied iff there exists an open set $W=\O(\tilde \rho,\tilde S,a) \cap Z'$ contained in $U$ such that the stabilizer of $a$ is the same for any two elements of $W$; that is, if and only if there exist finite sets $\tilde S_1$, $\tilde S_2 \subseteq G$, with $\tilde S_1 \supseteq S_1$ and $\tilde S_2 \supseteq S_2$, and a subgroup $G'_a$ such that:
$$\forall H \in \Sub(G) \ \left( \forall g \in S_1 \ g \in H \text{ and } \forall g \in S_2 \ g \not \in H \right) \Leftrightarrow  (H=G'_a)\ .$$  

Thus $H$ is an isolated point in $\Sub(G)$; since $G_a$ was an arbitrary subgroup of $G$, and $S_1,S_2$ encode an arbitrary open neighborhood of $G_a$, we just established that there exists a generic action in $\Hom(G,X!)$ if, and only if, $G$ is solitary.

One can try to use the same approach as above to understand when there exists generic conjugacy classes in $\Hom(G,P)$ for other Polish groups $P$. But, as the structure of $P$ becomes more complicated, the above analysis is harder to carry out (in particular, the reduction to transitive actions no longer works). One case when one can do it is when $P=\text{Aut}(\mathcal R)$ is the automorphism group of the random graph. Then, reasoning in much the same way as above, one obtains the following criterion, which probably can be further simplified.

\begin{proposition}\label{p:random}
Given $G$ a countable group, there exists a generic element in $\Hom(G,\text{Aut}(\mathcal R))$ iff the following condition is satisfied:

For any finitely-generated subgroups $H_1,\ldots,H_n$ of $G$, and any finite $K_{i,j} \subseteq \G$ with $H_iH_j \cap K_{i,j} = \emptyset$ for all $i \le j$, there exists finitely generated subgroups $H'_1,\ldots,H'_n$ such that:
\begin{itemize}
\item For all $i$ $H'_i$ contains $H_i$.
\item For all $i,j$ the double coset space $H'_i \backslash G /H'_j$ is finite.
\item For all $i \le j$ one has $H'_iH'_j \cap K_{i,j}= \emptyset$.
\end{itemize}
\end{proposition}

The above property is hard to grasp; it does imply that $G$ is finitely generated, and that any finitely generated subgroup of $G$ is an intersection of finitely generated subgroups with finite bi-index. Must a group satisfying the previous conditions be LERF?

\section{A-separability}\label{sec:Asep}

\begin{definition}  \label{def:amenable_action}
An action $G \acts{} X$ of a discrete, countable group $G$ is called \emph{amenable} if it satisfies any one of the following equivalent conditions:
\begin{itemize}
\item For every $\epsilon > 0$ and $\Omega\subset G$ finite, $X$ admits an \emph{$(\epsilon,\Omega)$-F\o lner} subset, that is, a finite set $F\subset X$ such that $\dfrac{|gF\Delta F|}{|F|} < \epsilon$ for all $g\in \Omega$.
\item There exists a finitely additive $G$-invariant probability measure on $X$. 
\end{itemize}
When the action is transitive, of the form $G \acts{} G/K$, these conditions are further equivalent to the following:
\begin{itemize}
\item If $G$ acts continuously on a compact space and $K$ admits an invariant Borel measure, then so does $G$. 
\end{itemize}
\end{definition}

\noindent In the transitive case it is sometimes convenient to adopt group theoretic terminology as follows:
\begin{definition} \label{def:coam}
A subgroup $K$ of a group $G$ is called \emph{co-amenable} if the quasiregular action $G \acts{} G/K$ is amenable.
\end{definition}
\noindent The equivalence of these three conditions is classical. By definition, an amenable action always admits a F\o lner-sequence. This is a sequence of finite subsets $F_n \subset X$ such that for all $g \in G$ we have  $ \lim_{n \to \infty} \limits \dfrac{\left| g F_n \triangle F_n \right|}{\left| F_n \right |} = 0$. We recall the following:

\begin{remark}
A F\o lner-sequence can be chosen to be increasing (with respect to inclusion).
\end{remark}

As defined in the introduction, a group $G$ is A-separable if the set of co-amenable subgroups is dense in $\Sub(G)$. We now prove Theorem \ref{Asep thm}, giving a characterization of A-separability in the language of generic actions:
\begin{proof}[Proof of Theorem \ref{Asep thm}]
Denote by $\Cam(G)$ the set of all co-amenable subgroups of $G$. Note  For $x\in X$ denote
\begin{align*}
\Sigma(x) &=\{\sigma\in\Hom(G,X!)\mid G\acts \sigma(G)x\text{ is amenable}\}\\
	       &=\{\sigma\in\Hom(G,X!)\mid G_x(\sigma)\in\Cam(G)\}.
\end{align*}
and $\Sigma = \cap_{x  \in X} \Sigma(x)$. If $\Sigma$ is generic then $\Sigma(x)$ is dense in $\Hom(G,X!)$ and by Lemma \ref{lem1} the image of this set $\{ G_x(\sigma)\mid \sigma\in\Sigma(x)\}$ is a dense subset of $\Sub(G)$ consisting of co-amenable subgroups.  

Conversely, assume that $\Cam(G)$ is dense in $\Sub(G)$ and we wish to prove that the set $\Sigma$ is generic in $\Hom(G,X!)$. It is enough to show that $\Sigma(x)$ is generic in $\Hom(G,X!)$ for every $x\in X$. The density of $\Sigma(x)$ is assured by the hypothesis, the fact that $\Sigma(x) = G_x^{-1}(\Cam(G))$ and Lemma \ref{lem1}. To show that $\Sigma(x)$ is $G_{\delta}$, it is enough to show that the condition that a specific finite set $F \subset \sigma(G)x$ is $(\epsilon,\Omega)$-F\o lner is open, where $\epsilon > 0$ and $\Omega \subset G$ is finite. Assume this is the case for some $\sigma \in \Hom(G,X!)$ we seek an open neighborhood $\sigma \in \mathcal{O} \subset \Hom(G,X!)$ such that $F$ is still contained in the orbit, and is still F\o lner for every $\phi \in \mathcal{O}$. For every $f \in F$ pick a group element $g_f\in G$ such that $f = \sigma(g_f)x$. Let $F' = F\cup \{x\}$ and $\Omega' = \Omega \cup \{g_f \ | \ f \in F\}$ - the desired neighborhood is given by $\mathcal{O} = \mathcal{O}(\sigma,F',\Omega')$. 
\end{proof}

As mentioned in the introduction, LERF groups and amenable groups are A-separable but they are not the only examples of A-separable groups. In order to give an example of an A-separable group which is neither LERF nor amenable, we will first prove that A-separability is closed under taking free products:

\begin{theorem}\label{Asep product}
Let $G$ and $K$ be countable groups. If $G$ and $K$ are A-separable then so is $G * K$.
\end{theorem}
\begin{proof}
Every element of $\Hom(G * K,X!)$ is of the form $\phi * \psi$ for $\phi\in\Hom(G,X!), \psi\in\Hom(K,X!)$; where $\phi * \psi$ is defined by setting $\big(\phi * \psi\big)(g) = \phi(g)$ and $\big(\phi * \psi\big)(k) = \psi(k), \forall g\in G, k \in K$ and expanding the definition to the free product.

For every $x\in X,\ \epsilon > 0$ and finite subsets $S\subset G,\ T\subset K$ let 
\begin{multline*}
\Sigma(x,\epsilon,S,T) =\{ \sigma*\tau\in\Hom(G * K,X!) \mid \text{the $(\sigma *\tau)$-orbit of $x$}
\\ \text{contains an $(\epsilon,S\cup T)$-F\o lner set}\}.
\end{multline*}
We want to prove that $\Sigma = \bigcap \Sigma(x,\frac{1}{n},S,T)\ \big( x\in X,n\in\N,S\subset G, T\subset K \text{ finite}\big)$ is generic in $\Hom(G * K,X!)$. Since $X, G$ and $K$ are countable, it is enough to show that the sets $\Sigma(x,\epsilon,S,T)$ are open and dense for every $x\in X,\ \epsilon > 0$ and finite subsets $S\subset G,\ T\subset K$. The argument that shows that $\Sigma(x,\epsilon,S,T)$ is open was given in the proof of Theorem \ref{Asep thm}.

Fix $x, \epsilon, S$ and $T$ as above. We prove that $\Sigma(x,\epsilon,S,T)$ is dense in $\Hom(G * K,X!)$. Let $\phi *\psi\in\Hom(G*K,X!)$ and let $A\subset X$ be finite. We will find $\phi' \in\Hom(G,X!)$ and $\psi' \in\Hom(K,X!)$ such that $\phi'(s)a = \phi(s)a,\ \psi'(t)a =\psi(t)a$ for all $s\in S, t\in T, a\in A$ and such that $\phi' *\psi' \in \Sigma(x,\epsilon,S,T)$. We can assume that $x\in A$. By A-separability, there exist $\sigma \in\Hom(G,X!)$ and $\tau \in \Hom(K,X!)$ such that $\sigma(s)a = \phi(s)a,\ \tau(t)a =\psi(t)a$ for all $s\in S, t\in T, a\in A$ and such that the actions $G \acts{\sigma} X$ and $K\acts{\tau}X$ are amenable on every orbit. Let $L := \sigma *\tau(G*K) = \la \sigma(G),\tau(K)\ra$.

\emph{Case 1: all the $\sigma$ and $\tau$ orbits which are contained in $Lx$ are finite}. Let $B\subset Lx$ be a finite, $\sigma$-invariant set containing $A\cap Lx$ and let $C = \bigcup_{b\in B} \tau(K)b$. We define a representation $\phi' \in \Hom(G,X!)$ by declaring every $c\in C \minus B$ and every element in the $\sigma$-orbit of $c$ to be a fixed point for $\phi'$ and on every other element of $X$, $\phi'(g)$ identifies with $\sigma(g)$ for all $g\in G$. Notice that since $B$ is $\sigma$-invariant, $\phi'(g)$ is well defined and acts the same as $\sigma(g)$ on $B$ for all $g\in G$. In particular, $\phi'(g)$ agrees with $\phi(g)$ on $A$. We have that $C$ is finite, invariant  under both $\phi'$ and $\tau$ and contains $x$. Setting $\psi' = \tau$, the $(\phi'*\psi')$-orbit of $x$ is finite so the orbit itself is an $(\epsilon,S\cup T)$-F\o lner set for $\phi'*\psi'$.

\emph{Case 2: $Lx$ contains either an infinite $\sigma$-orbit or an infinite $\tau$-orbit}. Assume, without loss of generality, that $Lx$ contains an infinite $\tau$-orbit $Y$. Denote $B = A\cup\big( \bigcup_{s\in S} \sigma(s)A\big)$ and let $F_n$ be an increasing F\o lner-sequence in $Y$ for the $\tau$-action. Since the sets $F_n$ are finite, none of them is $\tau$-invariant and so the F\o lner-sequence does not stabilize. This implies that $|F_n| \to \infty$ and in particular, $Y$ contains an $(\epsilon,T)$-F\o lner set $F$ such that $|F| > \dfrac{2(|B|+1)}{\epsilon}$. Now, let $z\in G*K$ be such that $(\sigma *\tau)(z)x\in F$ and such that $z$ is of minimal length with respect to the canonical presentation: $z = g_n k_n g_{n-1} k_{n-1} \cdots g_1 k_1\ (g_i\in G, k_j\in K, g_1,\dots,g_{n-1},k_2,\dots,k_n \neq 1)$. Denote $y = (\sigma *\tau)(z)x$. By Lemma \ref{lem2bis}, we can assume that $\sigma$ has infinitely many fixed points. In particular, there exits a set $C\subset X$ on which $\sigma(G)$ acts trivially, such that $|C| = |F\minus(B\cup\{y\})|$ and such that $C$ does not intersect the finite set $B\cup F\cup \trace{x}{(\sigma *\tau)(z)}$ where we think of $(\sigma *\tau)(z)$ as the word over $X!$ corresponding to the given presentation of $z$. Denote $D = F\minus(B\cup\{y\})$ and let $\xi \in X!$ be a permutation of order $2$ that takes $C$ bijectively onto $D$ and acts trivially on $X\minus (C\cup D)$. We define an action $\phi' \in\Hom(G,X!)$ by $\phi'(g) = \xi^{-1}\sigma(g)\xi$ for all $g\in G$. Since $\xi$ acts trivially on $B$ we have that $\forall s\in S,\ \forall a\in A: \phi'(s)a = \sigma(s)a = \phi(s)a$ and that every element of $D$ is fixed under $\phi'(s)$ for all $s\in S$. Hence:
\[ \forall s\in S: \dfrac{|\phi'(s)F\Delta F|}{|F|} \leq \dfrac{2|F\minus D|}{|F|} \leq \dfrac{2(|B|+1)}{|F|} < \epsilon.\]
Thus $F$ is $(\epsilon,S)$-F\o lner for $\phi'$ and $(\epsilon,T)$-F\o lner for $\psi':=\tau$ and thus $F$ is $(\epsilon,S\cup T)$-F\o lner for $\phi' * \psi'$. Notice that by the minimality of the length of $z$ we have that $\trace{x}{(\sigma *\tau)(z)}\cap F = \{y\}$ and so $\xi$ acts trivially on $\trace{x}{(\sigma *\tau)(z)}$. This means that $(\phi' *\psi')(z)x =y\in F$ and since $F$ is contained in a $\tau$-orbit this implies that $F$ is contained in the $(\phi' *\psi')$-orbit of $x$, as required.

\end{proof}
Recall that the $(m,n)$ Baumslag-Solitar group is the group $BS(m,n) = \la s,t\mid t^{-1}s^m t = s^n\ra$. It is well known that $BS(m,n)$ is solvable (hence amenable) if and only if $m=1$.

\begin{proposition}
For every $n$, the group $BS(1,n)$ is not LERF.
\end{proposition}
\begin{proof}
Write: $BS(1,n) = \la s,t\mid t^{-1}s t = s^n\ra$ and notice that $t^{-1}\la s\ra t = \la s^n\ra \subsetneqq \la s\ra$. Thus, an element of $\la s\ra \minus t^{-1}\la s\ra t$ cannot be separated from $t^{-1}\la s\ra t$ by a homomorphism into a finite group.
\end{proof}

\begin{corollary}
There exist non-LERF, non-amenable A-separable groups.
\end{corollary}
\begin{proof}
Let $G = BS(1,n)$ for some $n$. $G$ is amenable hence A-separable and so, by Proposition \ref{Asep product}, $G * G$ is A-separable. On the other hand, $G * G$ is not LERF since $G$ is not LERF and the LERF property passes to subgroups. $G * G$ is also not amenable since it contains a free subgroup on two generators.
\end{proof}

In order to complete the proofs of all the statements promised in the introduction we prove the following:
\begin{proposition}
A group $G$ with Kazhdan property (T) is A-separable if and only if it is LERF. In particular the following groups are never A-separable:
\begin{itemize}
\item Groups with property (T) that are not residually finite, and in particular any simple group with property (T). 
\item Irreducible lattices in higher rank semi-simple Lie groups with no compact factors that satisfy the congruence subgroup property. 
\end{itemize}
\end{proposition}
\begin{proof}
This follows directly from the fact that a transitive action $G \acts{} G/H$ is amenable if and only if $G/H$ is finite. The argument for that follows from property (T). If this action is amenable and $F \subset G/H$ is an $(K,\epsilon)$ F\o lner set then $1_{F} \in \ell^2(G/H)$ is a $(K,\epsilon)$-almost invariant vector. Taking $(K,\epsilon)$ to be Kazhdan constants for $G$ we can deduce the existence of a non-zero invariant vector $f \in \ell^2(G/H)$. Since the action of $G$ on $G/H$ is transitive $f$ must be a constant function. But a non-zero constant function is in $\ell^2$ if and only if $G/H$ is finite. 

Now if $\Gamma$ is a lattice as in the statement of the theorem it cannot be LERF because by the strong approximation theorem \cite[Window 9]{LS:SubgroupGrowth},\cite{NN:SA_in_group_theory} every Zariski dense subgroup has a finite index closure in the pro-congruence topology; which coincides with the profinite topology by assumption. By Kazhdan's theorem such a lattice has property (T) and the statement follows. 
\end{proof}
Note that conjecturally the congruence subgroup property automatically holds for such higher rank lattices, and this is indeed proven in many different cases. In particular the groups $\operatorname{SL}_n(\Z), \ n \ge 3$ are good examples for generically finite groups that is not A-separable. We conjecture further that the LERF property can never occur in a nontrivial way for property (T) groups namely:
\begin{conjecture}
A countable group $G$ with Kazhdan property (T) is LERF if and only if it is finite. 
\end{conjecture}
It was pointed out to us by Matthew Stover that A similar question was already asked by Long and Reid in \cite[Question 4.5]{LR:subgroup_sep}. We are grateful to some very constructive comments on the previous version of this paper by Mikl\'{o}s Ab\'{e}rt, Eli Glasner, and Matthew Stover. J.M wishes to thank Romain Tessera for interesting conversations and suggestions about solitary groups. 

This work was written while the first author was on sabbatical at the University of Utah. I am  grateful to the math department there for their hospitality. Y.G. acknowledges support from U.S. National Science Foundation grants DMS 1107452, 1107263, 1107367 ``RNMS: Geometric structures And Representation varieties" (the GEAR Network) that enabled this visit. The first and second authors were both partially supported by the Israel Science Foundation grant ISF 2095/15, and the third author was partially supported by 
Agence Nationale de la Recherche grant Grupoloco (ANR-11-JS01-0008).

\bibliographystyle{alpha}
\bibliography{yair}

\noindent {\sc Yair Glasner.} Department of Mathematics.
Ben-Gurion University of the Negev.
P.O.B. 653,
Be'er Sheva 84105,
Israel.
{\tt yairgl\@@math.bgu.ac.il}\bigskip

\noindent {\sc Daniel Kitroser.} Department of Mathematics.
Ben-Gurion University of the Negev.
P.O.B. 653,
Be'er Sheva 84105,
Israel.
{\tt kitrosar\@@post.bgu.ac.il}\bigskip

\noindent {\sc Julien Melleray.} 
  Universit\'e Claude Bernard -- Lyon 1 .
  Institut Camille Jordan, CNRS UMR 5208 . 
  43 boulevard du 11 novembre 1918, 
  69622 Villeurbanne Cedex, 
{\tt melleray\@@math.univ-lyon1.fr}\bigskip

\end{document}